\documentclass{amsart}
\usepackage{indentfirst,enumerate,mathrsfs}
\usepackage{array,amscd}
\usepackage[pdftex]{hyperref}
\usepackage{caption}
\usepackage{multirow}
\renewcommand{\arraystretch}{1.5}
\def\NN{\mathbb N}
\def\DD{\mathcal D}
\def\RR{\mathbb R}
\def\HH{\mathcal H}
\def\xx{\mathbf{x}}
\def\yy{\mathbf{y}}
\def\zz{\mathbf{z}}
\newcommand{\NT}{\mathcal{N}_{T_\nu}}
\newcommand{\NL}{\mathcal{N}_{L_\nu}}
\newcommand{\ga}{g_\lambda}
\newcommand{\ra}{r_\lambda}
\newcommand{\lh}{\mathcal{L}(\HH)}
\newcommand{\la}{\lambda}

\newcommand{\fz}{f_{\mathbf{z},\la}}
\newcommand{\uz}{u_{\mathbf{z},\la}}
\newcommand{\fp}{f_\rho}
\newcommand{\up}{u_\rho}
\newcommand{\fr}{f_\rho^R}
\newcommand{\gp}{g_\rho}
\newcommand{\LL}{\mathscr{L}^2(X,\nu;Y)}

\newcommand{\sx}{S_\xx}
\newcommand{\lx}{L_\xx}
\newcommand{\bx}{B_\xx}
\newcommand{\tx}{T_\xx}
\newcommand{\ip}{I_\nu}
\newcommand{\lp}{L_\nu}
\newcommand{\bp}{B_\nu}
\newcommand{\tp}{T_\nu}
\newcommand{\op}[1]{\operatorname{#1}}

\newcommand{\tsup}{\sup\limits_{t\in[0,\kappa^2]}}
\newcommand{\argmin}{\operatornamewithlimits{argmin}}
\newcommand{\paren}[1]{\left(#1\right)}
\newcommand{\brac}[1]{\left\{#1\right\}}
\newcommand{\sbrac}[1]{\left[#1\right]}
\newcommand{\norm}[2]{\left\|{#1}\right\|_{#2}}
\newcommand{\scalar}[3]{\langle{ #1},{#2} \rangle_{#3}}
\newcommand{\abs}[1]{\left\lvert #1 \right\rvert}
\newcommand{\tr}{\operatorname{tr}}
\newcommand{\range}{\mathcal R}
\newtheorem{theorem}{Theorem}[section]
\newtheorem{lemma}[theorem]{Lemma}
\newtheorem{corollary}[theorem]{Corollary}
\newtheorem{proposition}[theorem]{Proposition}
\theoremstyle{definition}
\newtheorem{definition}[theorem]{Definition}
\newtheorem{assumption}{Assumption}
\theoremstyle{remark}
\newtheorem{remark}[theorem]{Remark}
\newtheorem{example}[theorem]{Example}
\RequirePackage{color}

\allowdisplaybreaks
\begin{document}
\title{Inverse learning in Hilbert scales}  

\author{Abhishake Rastogi}
\address{Institute of Mathematics, University of Potsdam, Karl-Liebknecht-Strasse 24-25, 14476 Potsdam, Germany} 
\email{abhishake@uni-potsdam.de}

\author{Peter Math{\'e}}
\address{Weierstrass Institute for Applied Analysis and Stochastics, Mohrenstrasse 39, 10117 Berlin, Germany}  
\email{peter.mathe@wias-berlin.de}

\date{}

\thanks{This research has been partially funded by Deutsche Forschungsgemeinschaft (DFG) - SFB1294/1 - 318763901.}  



\begin{abstract}
We study the linear ill-posed inverse problem with noisy data in the statistical learning setting. Approximate reconstructions from random noisy data are sought with general regularization schemes in Hilbert scale. We discuss the rates of convergence for the regularized solution under the prior assumptions and a certain link condition. We express the error in terms of  certain distance functions. For regression functions with smoothness given in terms of source conditions the error bound can then be explicitly established.  
\end{abstract}

\maketitle
\section{Introduction}\label{Sec:Introduction}
Let~$A$ be a linear injective operator between the infinite-dimensional Hilbert spaces~$\HH$ and~$\HH'$ with the inner products~$\scalar{\cdot}{\cdot}{\HH}$ and~$\scalar{\cdot}{\cdot}{\HH'}$, respectively. Let~$\HH'$ be the space of functions between a Polish space~$X$ and a real separable Hilbert space~$Y$. Here we study the linear ill-posed operator problems governed by the operator equation
\begin{equation}\label{Model}
  A(f) = g,\qquad\text{for}\quad f\in \HH \quad \text{and} \quad g\in\HH'.
\end{equation}

We observe noisy values of~$g$ at some points, and the foremost objective is to estimate the true solution~$f$. The problem of interest can be described as follows: Given data~$\brac{(x_i,y_i)}_{i=1}^m$ under the model
\begin{equation}
  y_i= g(x_i)+\varepsilon_i,\quad  i=1,\ldots,m,
\end{equation} 
where~$\varepsilon_i$ is the observational noise, and~$m$ denotes the sample size, determine (approximately) the underlying element~$f\in\HH$ with~$g:= A(f)$ being the regression function.

For classical inverse problems, the observational noise is assumed to be deterministic. Here we assume that the random observations~$\brac{(x_i,y_i)}_{i=1}^m$ are independent and follow some unknown probability distribution~$\rho$, defined on the sample space~$Z=X\times Y$, and hence we are in the context of statistical inverse problems.

The reconstruction of the unknown true solution will be based on spectral regularization schemes. Various schemes can be used to stably estimate the true solution. Tikhonov regularization is widely-considered in the literature. This scheme consists of the error term measuring the fitness of the data and a penalty term, controlling the complexity of the reconstruction. In this study we enforce smoothness of the approximated solution by introducing an unbounded, linear, self-adjoint, strictly positive operator~$L: \DD(L) \subset \HH \to \HH$ with a dense domain of definition~$\DD(L) \subset \HH$, and then we define \emph{Tikhonov regularization scheme in Hilbert scales} as follows:
\begin{equation}\label{Tikhonov}
  \argmin\limits_{f\in\DD(L)}\brac{\frac{1}{m}\sum\limits_{i=1}^m\norm{ A(f)(x_i)-y_i}{Y}^2+\la\norm{L f}{\HH}^2},
\end{equation}
where~$\la$ is a positive regularization parameter and the operator~$L$ influences the properties of the approximated solution. Standard Tikhonov regularization corresponds to~$L:= Id\colon \HH \to \HH$, the identity mapping. In many practical problems, the operator~$L$ is chosen to be a differential operator in some appropriate function spaces, e.g.,~$\mathscr{L}^2$-spaces.

Notice from~(\ref{Tikhonov}), that the reconstruction~$\fz$ belongs to~$\DD(L)$, such that formally we may introduce~$\uz:= L \fz\in\HH$. In the regular case, when~$\fp \in \DD(L)$, then we let~$\up:= L\fp\in\HH$. With this notation we can rewrite~(\ref{Model}) as
$$
g = A f = A L^{-1} u,\quad u\in\DD(L).
$$
Also, the Tikhonov minimization problem would reduce to the standard one
$$
\argmin\limits_{u}\brac{\frac{1}{m}\sum\limits_{i=1}^m\norm{
    (AL^{-1})(u)(x_i)-y_i}{Y}^2+\la\norm{u}{\HH}^2},
$$
albeit for a different operator~$A L^{-1}$.  Accordingly, the error bounds relate as
$$
\norm{ \fp - \fz}{\HH} = \norm{L^{-1}(\up - \uz)}{\HH}.
$$
Therefore, error bounds for~$\up - \uz$ in the weak norm, in~$\HH_{-1}$, yield bounds for~$\fp - \fz$. The latter bounds are not known from previous studies. Also, we are interested in the oversmoothing case, when~$\fp \not\in\DD(L)$, such that we provide a detailed error analysis, here. However, the above relation will implicitly be utilized in the subsequent proofs.
 
We review literature related to the considered problem. Regularization schemes in Hilbert scales are widely considered in classical inverse problems (with deterministic noise), starting from F. Natterer~\cite{Natterer1984}, and continued in~\cite{Bottcher2006,Mair1994,Mathe2006,Mathe2007,Nair1999,Nair2002,Nair2005,Neubauer1988,Tautenhahn1996}. G. Blanchard and N.~M\"ucke~\cite{Blanchard} considered general regularization schemes for linear inverse problems in statistical learning and provided (upper and lower) rates of convergence under H{\"o}lder type source conditions. Here we consider general (spectral) regularization schemes in Hilbert scales for the statistical inverse problems.  We discuss rates of convergence for general regularization under certain noise conditions, approximate source conditions, and a specific link condition between the operators~$A$, governing the equation~(\ref{Model}), and the smoothness promoting operator~$L$ as used e.g. in~(\ref{Tikhonov}). We study error estimates by using the concept of reproducing kernel Hilbert spaces. The concept of the effective dimension plays an important role in the convergence analysis.

The key-points in our results can be described as follows:
\begin{itemize}
\item We do not restrict ourselves to the white or coloured Hilbertian noise. We consider general centered noise, satisfying certain moment conditions, see Assumption~\ref{Tikhonov}.

\item We consider general regularization schemes in Hilbert scales. It is well-known that Tikhonov regularization suffers the saturation effect. On the contrary, this saturation is delayed for Tikhonov regularization in Hilbert scales.

\item The analysis uses the concept of link conditions, see Assumption~\ref{ass:link}, required to transfer information in terms of properties of the operator~$L$ to the covariance operator.

\item We analyze the \emph{regular case}, i.e.,\ when the true solution belongs to the domain of operator~$L$.

\item We also focus on the \emph{oversmoothing case}, when the true solution does not belong to the domain of operator~$L$.
\end{itemize}

The paper is organized as follows. The basic definitions, assumptions, and notation required in our analysis are presented in Section~\ref{Sec-Notation}. In Section~\ref{Sec-analysis} we discuss the bounds of the reconstruction error in the direct learning setting and inverse problem setting by means of distance functions. This section comprises of two main results: The first result is devoted to convergence rates in the oversmoothing case,  while the second result focuses on the regular case. When specifying smoothness in terms of source conditions we can bound the distance functions, and this gives rise to convergence rates in terms of the sample size~$m$. This program is performed in Section~\ref{sec:smoothness}. In case that both, the smoothness as well as the link condition are of power type we establish the optimality of the obtained error bounds in the regular case in Section~\ref{sec:optimality}. In the Appendix, we present probabilistic estimates which provide the tools to obtain the error bounds. 

\section{Notation and Assumptions}\label{Sec-Notation}
In this section, we introduce some basic concepts, definitions, notation, and assumptions required in our analysis. 

We assume that~$X$ is a Polish space, therefore the probability distribution~$\rho$ allows for a disintegration as
\begin{equation*}
  \rho(x,y)=\rho(y|x)\nu(x),
\end{equation*}
where~$\rho(y|x)$ is the conditional probability distribution of~$y$ given~$x$, and~$\nu(x)$ is the marginal probability distribution. We consider random observations~$\brac{(x_i,y_i)}_{i=1}^m$ which follow the model~$y= A(f)(x)+\varepsilon$ with centered noise~$\varepsilon$. We assume throughout the paper that the operator~$A$ is injective. 
\begin{assumption}[The true solution]\label{fp}
The conditional expectation w.r.t.~$\rho$ of~$y$ given~$x$ exists (a.s.), and there exists~$\fp \in \HH~$ such that
  \begin{equation*}
\int_Y y d\rho(y|x)= \gp(x) = A(\fp)(x), \text{ for all } x\in X.
  \end{equation*}
\end{assumption} 

The element~$\fp$ is the true solution which we aim at estimating.

\begin{assumption}[Noise condition]\label{noise.cond}
There exist some constants~$M,\Sigma$ such that for almost all~$x\in X$,
\begin{equation*}
\int_Y\left(e^{\norm{y-A(\fp)(x)}{Y}/M}-\frac{\norm{y-A(\fp)(x)}{Y}}{M}-1\right)d\rho(y|x)\leq\frac{\Sigma^2}{2M^2}.
\end{equation*}
\end{assumption}

This assumption is usually referred to as a \emph{Bernstein-type assumption}.

We return to the unbounded operator~$L$. By spectral theory, the operator~$L^s : \DD(L^s) \to \HH$ is well-defined for~$s \in \RR$, and the spaces~$\HH_s := \DD(L^s), s \geq 0$ equipped with the inner product~$\scalar{f}{g}{\HH_s}=\scalar{ L^s f}{L^s g}{\HH},\quad f, g \in\HH_s$ are Hilbert spaces. For~$s < 0$, the space~$\HH_s$ is defined as completion of~$\HH$ under the norm~$\norm{f}s := \scalar{f}{f}{s}^{1/2}$. The space~$(\HH_s)~s\in\RR$ is called the Hilbert scale induced by~$L$. The following interpolation inequality is an important tool for the analysis:
\begin{equation}\label{interpolation}
  \norm{f}{\HH_r}\leq\norm{f}{\HH_t}^{\frac{s-r}{s-t}}\norm{f}{\HH_s}^{\frac{r-t}{s-t}},\qquad  f\in \HH_s, 
\end{equation}
which holds for any~$t < r < s$~\cite[Chapt.~8]{Engl}.

\subsection{Reproducing Kernel Hilbert space and related operators}
\label{sec:repr}
We start with the concept of reproducing kernel Hilbert spaces. It is a subspace of~$\LL$ (the space of square-integrable functions from~$X$ to~$Y$ with respect to the probability distribution~$\nu$) which can be characterized by a symmetric, positive semidefinite kernel and each of its functions satisfies the reproducing property. Here we discuss the vector-valued reproducing kernel Hilbert spaces, following~\cite{Micchelli1}, which are the generalization of real-valued reproducing kernel Hilbert spaces~\cite{Aronszajn}.

\begin{definition}[Vector-valued reproducing kernel Hilbert space]
For a non-empty set~$X$ and a real separable Hilbert space~$(Y,\scalar{\cdot}{\cdot}{Y})$, a Hilbert space~$\HH$ of functions from~$X$ to~$Y$ is said to be the vector-valued reproducing kernel Hilbert space, if the linear functional~$F_{x,y}:\HH \to \RR$, defined by
$$F_{x,y}(f)=\scalar{ y}{f(x)}{Y} \qquad \forall f \in \HH,$$
is continuous for every~$x \in X$ and~$y\in Y$.
\end{definition}

\begin{definition}[Operator-valued positive semi-definite kernel]
Suppose~$\mathcal{L}(Y):Y\to Y$ is the Banach space of bounded linear operators. A function~$K:X\times X\to \mathcal{L}(Y)$ is said to be an operator-valued positive semi-definite kernel if
\begin{enumerate}[(i)]

\item~$K(x,x')^*=K(x',x) \qquad\forall~x,x'\in X.$
  
\item~$\sum\limits_{i,j=1}^N\scalar{ y_i}{K(x_i,x_j)y_j}{Y}\geq 0 \qquad\forall~\{x_i\}_{i=1}^N\subset X \text{ and } \{y_i\}_{i=1}^N\subset Y.$

\end{enumerate}
\end{definition}

For a given operator-valued positive semi-definite kernel~$K:X \times X \to \mathcal{L}(Y)$, we can construct a unique vector-valued reproducing kernel Hilbert space~$(\HH,\scalar{\cdot}{\cdot}{\HH})$ of functions from~$X$ to~$Y$ as follows:
\begin{enumerate}[(i)]
\item We define the linear function
  \[
    K_x: Y \rightarrow \HH: y \mapsto K_xy,
  \]
  where~$K_xy:X \to Y:x' \mapsto (K_xy)(x')=K(x',x)y$ for~$x,x'\in X$ and~$y\in Y$.
\item The span of the set~$\{K_xy:x\in X, y\in Y\}$ is dense in~$\HH$.
\item \textbf{Reproducing property:}
  \[\scalar{f(x)}{y}{Y}=\scalar{ f}{K_xy}{\HH},\qquad x\in  X,~y \in Y,~\forall~f\in \HH,\] in other words~$f(x) = K_x^* f$.
\end{enumerate}

Moreover, there is a one-to-one correspondence between operator-valued positive semi-definite kernels and vector-valued reproducing kernel Hilbert spaces, see~\cite{Micchelli1}.
  
We assume the following assumption concerning the Hilbert space~$\HH'$:
\begin{assumption} \label{assmpt1} The space~$\HH'$ is assumed to be a vector-valued reproducing kernel Hilbert space of functions~$f:X\to Y$ corresponding to the kernel~$K:X\times X\to \mathcal{L}(Y)$ such that
\begin{enumerate}[(i)]
  \item~$K_x:Y\to\HH'$ is a Hilbert-Schmidt operator for~$x\in X$ with
    \[\kappa'^2:=\sup_{x \in X} \norm{K_x}{HS}^2 = {\sup_{x \in
          X}\tr(K_x^*K_x)}<\infty.\]
  \item For~$y,y'\in Y$, the real-valued function~$\varsigma:X\times X \to \RR:(x,x')\mapsto\scalar{ K_{x}y}{K_{x'}y'}{\HH'}$ is measurable.
\end{enumerate}
\end{assumption}

\begin{example}
In case that the set~$Y$ is a bounded subset of~$\RR$ then the reproducing kernel Hilbert space becomes real-valued reproducing kernel Hilbert space. The corresponding kernel becomes the symmetric, positive semi-definite~$K:X \times X \to \RR$ with the reproducing property~$f(x)=\scalar{ f}{K_x}{\HH}$.  Also, in this case the Assumption~\ref{assmpt1} simplifies to the condition that the kernel is measurable and~$\kappa'^2:=\sup_{x \in X} \norm{K_x}{\HH'}^2=\sup_{x \in X}K(x,x)<\infty$. 
\end{example}
 
Now we introduce some relevant operators used in the convergence analysis. We introduce the notation for the vectors~$\xx=(x_1,\ldots,x_m)$,~$\yy=(y_1,\ldots,y_m)$,~$\zz=(z_1,\ldots,z_m)$. The product Hilbert space~$Y^m$ is equipped with the inner product~$\scalar{\yy}{\yy'}{m} = \frac{1}{m}\sum_{i=1}^m \scalar{y_i}{y'_i}{Y},$ and the corresponding norm~$\norm{\yy}{m}^2=\frac{1}{m}\sum_{i=1}^m\norm{y_i}{Y}^2$. We define the {\it sampling operator}~$\sx:\HH' \to Y^m:g\mapsto(g(x_1),\ldots,g(x_m))$, then the adjoint~$\sx^*:Y^m\to\HH'$ is given by
$$\sx^*\yy=\frac{1}{m}\sum_{i=1}^m K_{x_i} y_i.$$  

Let~$\ip$ denotes the canonical injection map~$\HH' \to \LL$. Then we observe that, under Assumption~\ref{assmpt1},  both the operators~$\sx$ and~$\ip$ are bounded by~$\kappa'$, since
\[\norm{\ip f}{\LL}^2=\int_X\norm{f(x)}{Y}^2d\nu(x) =\int_X\norm{K_x^*f}{Y}^2d\nu(x) \leq \kappa'^2\norm{f}{\HH}^2
\]
and
\[ \norm{\sx f}{m}^2 =\frac{1}{m}\sum_{i=1}^m\norm{f(x_i)}{Y}^2 =\frac{1}{m}\sum_{i=1}^m\norm{K_{x_i}^*f}{Y}^2 \leq
  \kappa'^2\norm{f}{\HH}^2.\]

We denote the population operators~$\bp := \ip AL^{-1}:\HH\to\LL$,~$\tp:= \bp^*\bp:\HH\to\HH$,~$\lp:= A^*\ip^*\ip A:\HH\to\HH$, and their empirical versions~$\bx=\sx A L^{-1}:\HH\to Y^m$,~$\tx=\bx^*\bx:\HH\to\HH$,~$\lx=A^*\sx^*\sx A:\HH\to\HH$. The operators~$\tp$,~$\tx$,~$\lp$,~$\lx$ are positive, self-adjoint and depend on the kernel. Under Assumption~\ref{assmpt1}, the operators~$\bx$,~$\bp$ are bounded by~$\kappa:=\kappa' \norm{AL^{-1}}{\HH\to\HH'}$ and the operators~$\lx$,~$\lp$ are bounded by~$\tilde{\kappa}^2$ for~$\tilde{\kappa}:=\kappa' \norm{A}{\HH\to\HH'}$,  i.e.,~$\norm{\bx}{\HH \to Y^m}\leq\kappa$,~$\norm{\bp}{\HH \to \LL}\leq\kappa$,~$\norm{\lx}{\lh}\leq\kappa^2$ and~$\norm{\lp}{\lh}\leq\tilde{\kappa}^2$.

\subsection{Link condition}\label{Sec:link.cond}
In the subsequent analysis, we shall derive convergence rates by using \emph{approximate source conditions}, which are related to a certain benchmark smoothness. This benchmark smoothness is determined by the user. In order to have handy arguments to derive the convergence rates, we shall fix an (integer)  power~$q\geq 1$. We shall use a link condition to transfer smoothness in terms of the operator L to the covariance operator~$\tp$. This link condition will involve an index function.
\begin{definition}[Index function]
A function~$\varphi : \RR^+ \to \RR^+$ is said to be an index function if it is continuous and strictly increasing with~$\varphi(0) = 0$. 
\end{definition}
An index function is called sub-linear whenever the mapping~$t\mapsto t/\varphi(t),\ t>0,$ is nondecreasing. Further, we require this index function to belong to the following class of functions.
\begin{align}\label{fun.class}
\mathcal{F}=\{&\varphi=\varphi_1\varphi_2:\varphi_1,\varphi_2:[0,\kappa^2]\to[0,\infty),\varphi_1~\text{nondecreasing continuous sub-linear},\\  \nonumber
&\varphi_2~\text{ nondecreasing Lipschitz},~\varphi_1(0)=\varphi_2(0)=0\}.
\end{align}

The representation~$\varphi=\varphi_2\varphi_1$ is not unique, therefore~$\varphi_2$ can be assumed as a Lipschitz function with Lipschitz constant~$1$. Now we phrase an important result, needed in our analysis \cite[Corollary~1.2.2]{Peller}: 
$$
\norm{\varphi_2(\tx)-\varphi_2(\tp)}{HS}\leq \norm{\tx-\tp}{HS}.
$$
\begin{example}
 The polynomial function~$\varphi(t)=t^r$, and the logarithm function~$\varphi(t)=t^p\log^{-\nu}\left(\frac{1}{t}\right)$ are examples of functions in the class~$\mathcal{F}$.   
\end{example}

\begin{assumption}
  [link condition]\label{ass:link}
There exist a power~$q > 1$ and an index function~$\varrho$, for which the function~$\varrho^{2}$ is sub-linear. There are constants~$1 \leq \beta <\infty$ such that  
\begin{equation*}
\norm{L^{-q}u}{\HH}\leq \norm{\varrho^{q}(\tp)u}{\HH} \leq \beta^{q}\norm{L^{-q}u}{\HH},\quad u\in\HH.
\end{equation*}
The function~$t\mapsto \varphi(t):=\varrho^{q-1}(t)$ belongs to the class~$\mathcal{F}$.
\end{assumption}
As shown in~\cite{Bottcher2006}, Assumption~\ref{ass:link} implies the range identity~$\range(L^{-q}) =  \range(\varrho^{q}(\tp))$. In the context of a comparison of operators we mention the well-known Heinz Inequality, see~\cite[Prop.~8.21]{Engl}, which asserts that a comparison~$\norm{Gu}{\HH}\leq \norm{Hu}{\HH},\ u\in\HH$, for non-negative self-adjoint operators~$G,H\colon \HH \to \HH$ yields for every exponent~$0< q \leq 1$ that~$\norm{G^{q}u}{\HH}\leq \norm{H^{q}u}{\HH},\ u\in\HH$. Applying this to the above link condition we obtain the following: 
\begin{proposition}
  \label{prop:Heinz}
  Under Assumption~\ref{ass:link} we have
  \begin{align*}
\norm{L^{-1}u}{\HH}\leq \norm{\varrho(\tp)u}{\HH} \leq  \beta\norm{L^{-1}u}{\HH},\quad u\in\HH\\\intertext {and} \norm{L^{-(q-1)}u}{\HH}\leq \norm{\varrho^{q-1}(\tp)u}{\HH} \leq \beta^{(q-1)}\norm{L^{-(q-1)}u}{\HH},\quad u\in\HH.
  \end{align*}
  Moreover, we have that
\begin{equation}
  \label{eq:rho12-bound}
  \norm{\varrho(\tp) \paren{\tp + \la I}^{-1/2}}{\lh}\leq \frac{\varrho(\la)}{\sqrt\la},\quad 0 < \la \leq 1.
\end{equation}
\end{proposition}
\begin{proof}
The first assertions are a consequence of Heinz Inequality. For the last one, we argue as follows. Since~$\varrho^{2}$ is assumed to be sub-linear. Hence we find that
\begin{align*}
  \norm{\varrho(\tp) \paren{\tp + \la I}^{-1/2}}{\lh}
  & = \frac 1 {\sqrt\la} \norm{\varrho(\tp) \paren{\la\paren{\tp + \la I}^{-1}}^{1/2}}{\lh}\\
    &\leq  \frac 1 {\sqrt\la}  \norm{\varrho^{2}(\tp) \paren{\la\paren{\tp+\la I}^{-1}}}{\lh}^{1/2}\\
  & \leq \frac {\varrho(\la)} {\sqrt\la},
\end{align*}
which completes the proof.
\end{proof}
\begin{remark}
From the assertion, it is heuristically clear that the function~$\varrho^{2}$ cannot increase faster than linearly, because the operator~$\tp = L^{-1} \lp L^{-1}$ has~$L^{-2}$ in it. More details will be given in Section~\ref{sec:optimality}.
\end{remark}

Link conditions as in Assumption~\ref{ass:link} imply decay rates for the singular numbers of the operators, known as Weyl's Monotonicity Theorem~\cite[Cor.~III.2.3]{Bhatia1997}. In our case, this yields that~$s_{j}(\varrho(\tp))  = \varrho(s_{j}(\tp))\asymp s_{j}(L^{-1})$. For classical spaces, as e.g.\ Sobolev spaces, when~$L:= (-\Delta)^{-1/2}$, then~$s_{j}(L^{-1}) \asymp 1/j$ (one spatial dimension). For the above index function~$\varrho$ this means that~$s_{j}(\tp) \asymp \varrho^{-1}(1/j)$.
\begin{example}[Finitely smoothing]
In case that the function~$\varrho$, and hence its inverse is of power type then this implies a power type decay of the singular numbers of~$\tp$. In this case, the operator~$\tp$ is called finitely smoothing. \end{example}
\begin{example}[Infinitely smoothing]
If, on the other hand, the function~$\varrho$ is logarithmic, as e.g.,\ $\varrho(t)=\paren{\log\frac{1}{t}}^{-\frac{1}{\mu}}$, then~$s_{j}(\tp) \asymp e^{-j^{\mu}}$. In this case, the operator~$\tp$ is called infinitely smoothing. 
\end{example}

\subsection{Effective dimension}\label{sec:eff-dim}

Now we introduce the concept of the effective dimension which is an important ingredient to derive the rates of convergence under H{\"o}lder's source condition \cite{Blanchard,Caponnetto,Guo} and general source condition \cite{Lu2020,Rastogi}.  The effective dimension for the trace--class operator~$\tp$ is defined as,
$$\NT(\la):=\op{Tr}\left((\tp+\la I)^{-1}\tp\right), \text{  for }\la>0.$$
It is known that the function~$\la\to \NT(\la)$ is continuous and decreasing from~$\infty$ to zero for~$0 < \la < \infty$ for an infinite dimensional operator~$\tp$ (see for details \cite{Blanchard2012,Blanchard2020,Lin2015,Lu2020,Zhang}). 

The integral operator~$\tp$ is a trace class operator, hence the effective dimension is finite,  and we have that
$$
\NT(\la)\leq \norm{(\tp+\la I)^{-1}}{\lh}\op{Tr}\left(\tp\right) \leq \frac{\kappa^2}{\la}.
$$

In the subsequent analysis, we shall need a relationship between the effective dimensions~$\NT(\la)$ and~$\NL(\la)$. For this, the link condition (Assumption~\ref{ass:link}) is crucial. The arguments will be based on operator monotonicity and concavity.  Below, for an operator~$T$ we assign~$s_j(T),\ j=1,2,\dots$ the singular numbers of the operator~$T$. 

The following assumption was introduced in~\cite{Lin2015}. There, it was shown that it is satisfied for both moderately ill-posed and severely ill-posed operators.
\begin{assumption}\label{Ass.abs}
There exists a constant~$C$ such that for~$0 < t \leq \norm{\lp}{\mathcal L(\HH)}$ we have
$$t^{-1}\sum\limits_{s_j(\lp)< t}s_j(\lp)<C\#\brac{j, \quad s_j(\lp)\geq t}.$$
\end{assumption}
The relation between the effective dimensions is established in the following proposition, with proof will given in Appendix~\ref{sec:lemma-proof}.
\begin{proposition}\label{prop:relation.eff_dim}
Suppose Assumptions~\ref{ass:link} and~\ref{Ass.abs} hold true. Suppose the function~$\varrho$ from the link condition is such that the function~$t\mapsto \paren{\varrho^{2q}}^{-1}(t)$ is operator concave, and that there is some~$n\in\NN$ for which the function~$t\mapsto \varrho^{-1}(t)/t^n$ is concave. Then, there is~$\widetilde C$ for which we have that
\begin{equation*}
 \NL\paren{\frac{\la}{\varrho^2(\la)}} \leq  2\beta^{n+1}\widetilde{C}\NT(\la),\quad 0<\la\leq\norm{\tp}{\mathcal L(\HH)}.
\end{equation*}
\end{proposition}
\begin{remark}
For a power type function~$\varrho(t):= t^{a}$ the above concavity assumptions hold true whenever~$2aq \geq 1$ and~$n \leq 1/a \leq n+1$. In particular the number~$n$ is uniquely determined.
\end{remark}

\subsection{Regularization Schemes}\label{sec:Regu.Schemes}
General regularization schemes were introduced and discussed in ill-posed inverse problems and learning theory (See \cite[Section~2.2]{Lu.book} and \cite[Section~3.1]{Bauer} for brief discussion). By using the notation from~\S~\ref{sec:repr}, the Tikhonov regularization scheme from~(\ref{Tikhonov}) can be re-expressed as follows: 
\begin{equation*}
  \fz = \argmin\limits_{f\in\DD(L)}\brac{\norm{\sx A(f)-\yy}{m}^2+\la\norm{L f}{\HH}^2},
\end{equation*}
and its minimizer is given by
\begin{equation*}
  \fz=L^{-1}(\tx+\la I)^{-1}\bx^*\yy.
\end{equation*} 

We consider the following definition.
\begin{definition}[General regularization]\label{regularization}
We say that a family of functions~$\ga:[0,\kappa^2]\to\RR$,~$0<\la\leq a$, is a regularization scheme if there exists~$D,B,\gamma$ such that
  \begin{itemize}
  \item~$\tsup\abs{t \ga(t)}\leq D$.
  \item~$\tsup\abs{\ga(t)}\leq\frac{B}{\la}$.
  \item~$\tsup\abs{\ra(t)}\leq \gamma \qquad   \text{for}\quad \ra(t)=1-\ga(t)t$.
  \item For some constant~$\gamma_p$ (independent of~$\la$), the maximal~$p$ satisfying the condition:
$$\tsup\abs{\ra(t)}t^p\leq\gamma_p\la^p$$ 
is said to be the qualification of the regularization scheme~$\ga$.
\end{itemize}
\end{definition}

\begin{definition}
The qualification~$p$ covers the index function~$\varphi$ if the function~$t\to\frac{t^p}{\varphi(t)}$ is nondecreasing.
\end{definition}

We mention the following result.
\begin{proposition}\label{prop:regularization}
Suppose~$\varphi$ is a nondecreasing index function and the qualification, say~$p\geq 1$, of the regularization~$\ga$ covers~$\varphi$. Then
$$
  \tsup\abs{\ra(\sigma)}\varphi(\sigma)\leq c_p\varphi(\la),\quad c_p=\max(\gamma,\gamma_p).
$$
Also, we have that
$$
\tsup\abs{\ra(\sigma)}\varphi(\la + \sigma)\leq 2^{p}c_p\varphi(\la).
$$
\end{proposition}
\begin{proof}
The first assertion is a restatement of~\cite[Proposition~3]{Mathe}. For the second assertion, we stress that~$(\la + \sigma)^{p} \leq 2^{p-1} (\la^{p} + \sigma^{p})$, which follows from convexity. This yields
$$
\abs{\ra(\sigma)}\varphi(\la + \sigma)\leq \abs{\ra(\sigma)}(\la + \sigma)^{p}\frac{\varphi(\la + \sigma)}{(\la + \sigma)^{p}}\leq 2^{p-1}\abs{\ra(\sigma)}(\la^{p} + \sigma^{p})\frac{\varphi(\la)}{\la^{p}} \leq 2^{p}c_{p}\la^{p}\frac{\varphi(\la)}{\la^{p}},
$$
which implies the second assertion and completes the proof.
\end{proof}

Essentially all the linear regularization schemes (Tikhonov regularization, Landweber iteration or spectral cut-off) satisfy the properties of general regularization. Inspired by the representation for the minimizer of the Tikhonov functional we consider a general regularized solution in Hilbert scales corresponding to the above regularization in the form
\begin{equation}\label{fzl}
  \fz=L^{-1}\ga(\tx)\bx^*\yy.
\end{equation}

\section{Convergence analysis}\label{Sec-analysis}
Here we study the convergence for general regularization schemes in the Hilbert scale of the linear statistical inverse problem based on the prior assumptions and the link condition.

The analysis will distinguish between two cases, the `regular' one, when~$\fp\in\DD(L)$, and the `low smoothness' case, when~$\fp\not\in\DD(L)$. In either case, we shall first utilize the concept of \emph{distance functions}. This will later give rise to establish convergence rates in a more classical style.

For the asymptotical analysis, we shall require the standard assumption relating the sample size~$m$ and the parameter~$\la$ such that
\begin{equation}\label{l.la.condition}
\NT(\la) \leq  m\la \qquad \text{and}\qquad 0<\la\leq 1. 
\end{equation}

It will be seen, that asymptotically the condition~(\ref{l.la.condition}) is always satisfied for the parameter which is optimally chosen under known smoothness.

The fact that~$\NT(\la)$ is decreasing function of~$\la$ and~$\la\leq 1$ implies that~$\NT(1)\leq \NT(\la)$. Hence from condition~\eqref{l.la.condition} we obtain,
\begin{equation}\label{nl}
\NT(1) \leq  m\la.
\end{equation}

Several probabilistic quantities will be used to express the error bounds. Precisely, for an index function~$\zeta$ we let
\begin{align}
  \Xi^{\zeta}=\Xi^{\zeta}(\la) &:= \norm{\paren{\frac{1}{\zeta}}(\tx+\la I)\zeta(\tp+\la  I)}{\lh},\label{ali:xizeta}\\
\Lambda= \Lambda(\la)&: =\norm{(\lp+\la I)^{-1/2}(\lp-\lx)}{HS},\label{ali:lambda}\\ 
\Upsilon= \Upsilon(\la)&: =\norm{(\tp+\la I)^{-1/2}(\tp-\tx)}{HS}\label{ali:upsilon}, 
\intertext{and}
 \Psi= \Psi(\la)&:= \norm{(\tp+\la I)^{-{1/2}}\bx^*(\sx A\fp-\yy)}{\HH}.\label{ali:psi}
\end{align}
In case that~$\zeta(t) = t^{r}$ we abbreviate~$\Xi^{t^{r}}$ by~$\Xi^{r}$ and~$\Xi^{t}$ by~$\Xi$, not to be confused with the power. High probability bounds for these quantities are known, and these will be given correspondingly in Propositions~\ref{main.bound} and~\ref{I1}.

\subsection{The oversmoothing case}\label{sec:oversmoothing}
As mentioned before, we shall use distance functions, and these are called `approximate source conditions' sometimes, because these measure the violation of a benchmark smoothness. Here the benchmark will be~$\fp \in\DD(L)$. 
\begin{definition}[Approximate source condition]\label{app.source.cond}
We define the distance function~$d : [0, \infty)\to[0, \infty)$ by
\begin{align}\label{eq.app.source}
d(R)=\inf\brac{\norm{\fp - f}{\HH}:f= L^{-1}v \text{ and }\norm{v}{\HH} \leq R},\quad R>0.
\end{align}
We denote~$\fr$ the element which realizes the above minimization problem.
\end{definition}
Notice the following: If~$\fp\in\DD(L)$ then for some~$R$ the minimizer~$\fr$ of the distance function will obey~$\fr=\fp$.

\begin{remark}
In a rudimentary form, this approach was given in \cite[Thm.~6.8]{Baumeister1987}. It was then introduced in regularization theory in \cite{Hofmann2006}. Within learning theory, such a concept was also used in the study \cite{Smale2003}.
\end{remark}

\begin{theorem}\label{err.upper.bound.gen.l}
Let~$\zz$ be i.i.d. samples drawn according to the probability measure~$\rho$. Suppose the Assumptions~\ref{fp}--\ref{Ass.abs} hold true. Suppose that the qualification~$p$ of the regularization~$\ga$ covers the function~$\varrho$ (for~$\varrho(t)$ from Assumption~\ref{ass:link})  and that~$\varrho^{-1}(t)/t^n$,~$\paren{\varrho^{2q}}^{-1}(t)$ are concave, or operator concave functions for some~$n\geq 1$, respectively. Then for all~$0<\eta<1$, and for~$\la$ satisfying the condition~\eqref{l.la.condition} the following upper bound holds for the regularized solution~$\fz$ (\ref{fzl}) with confidence~$1-\eta$:
$$
  \norm{\fz-\fp}{\HH}\leq C\brac{d(R)+ 2R \varrho\paren{\la}}\log^4\paren{\frac{4}{\eta}},\quad R\geq \Sigma+\kappa M/\NT(1),
$$
where~$C$ depends on~$B$,~$D$,~$c_p$,~$\kappa$,~$n$,~$\beta$,~$\widetilde{C}$.
\end{theorem}

\begin{proof}
For the minimizer~$\fr$ of the distance function defined in \eqref{eq.app.source}, the error can be expressed as follows:
  \begin{align*}
    \fp-\fz= & L^{-1}\brac{\ra(\tx)L(\fp-\fr)+\ra(\tx)L\fr+\ga(\tx)\bx^*(\sx A \fp-\yy)}.
  \end{align*} 

By using Proposition~\ref{prop:Heinz} the error for the regularized solution can be bounded as
\begin{align}\label{over.bd3}
  \norm{\fp-\fz}{\HH}\leq & \norm{L^{-1} \ra(\tx)L(\fp-\fr)}{\HH}+\norm{L^{-1} \ra(\tx)L\fr}{\HH}+\norm{L^{-1} \ga(\tx)\bx^*(\sx A \fp-\yy)}{\HH}\\  \nonumber
  \leq & d(R)\underbrace{\norm{L^{-1} \ra(\tx)L}{\lh}}_{I_{1}}
         +\underbrace{\norm{\varrho(\tp) \ra(\tx)L\fr}{\HH}}_{I_{2}}
         +\underbrace{\norm{\varrho(\tp) \ga(\tx)\bx^*(\sx A
         \fp-\yy)}{\HH}}_{I_{3}}.
\end{align}
We shall bound each summand on the right in~(\ref{over.bd3}).
\begin{description}
\item[$I_{1}$] By Lemma~\ref{eq:lemma} we find that
$$
\norm{L^{-1}\ra(\tx) L}{\lh} \leq 1+(B + D) \paren{\Xi^\varrho  \Xi^\upsilon   + \Xi \varrho(\la)(\varrho(\la)+1)\frac{\Lambda}{\sqrt{\la}}}
$$
with~$\Xi^{\varrho}$,~$\Lambda$ as in~(\ref{ali:xizeta}),~(\ref{ali:lambda}) and~$\upsilon(t):= t/\varrho(t),\ t>0$.

From the estimates of Propositions~\ref{main.bound},~\ref{I1} we get with confidence~$1-\eta/2$ that
\begin{align}\label{over.bound2}
  \norm{L^{-1} \ra(\tx)L}{\lh}
  \leq &
         1+(B+D)\brac{(2\kappa+1)^8+2(2\kappa+1)^4(\varrho(\la)+1)\paren{\frac{\tilde{\kappa}\varrho(\la)}{m\la}+\sqrt{\frac{\tilde{\kappa}\varrho^2(\la)\NL(\la)}{m\la}}}}\\
  &\times \log^4\paren{\frac{4}{\eta}},\notag
\end{align}
For~$\vartheta(\la):=\frac{\la}{\varrho^2(\la)}$ under the fact that~$\la\NL(\la)$ is increasing function and~$\la\leq \vartheta(\la)$, for~$\la$ small enough, we get
\begin{equation*}
\la \NL(\la) \leq \vartheta(\la) \NL\paren{\vartheta(\la)}.
\end{equation*}
This together with Proposition~\ref{prop:relation.eff_dim} implies that
\begin{equation}\label{NL.link}
\varrho^2(\la)\NL(\la) \leq \NL\paren{\frac{\la}{\varrho^2(\la)}} \leq 2\beta^{n+1}\widetilde{C}\NT(\la).  
\end{equation}
Under the condition~\eqref{l.la.condition} from the estimates~\eqref{nl},~\eqref{over.bound2},~\eqref{NL.link} we get with confidence~$1-\eta/2$:
\begin{align}\label{over.bd2}
  \norm{L^{-1} \ra(\tx)L}{\lh}
  \leq & 1+(B+D)\beta^{n+1}\widetilde{C}C_{\kappa,\tilde{\kappa}}\log^4\paren{\frac{4}{\eta}},
\end{align}
where~$C_{\kappa,\tilde{\kappa}}$ depends on~$\kappa,\tilde{\kappa}$.
\item[$I_{2}$]
 By construction of~$\fr$ we have that~$\fr = L^{-1}v,\
  \norm{v}{\HH}\leq R$. Using the fact that~$p$ covers~$\varrho$  we bound
\begin{align}\label{err2}
\norm{\varrho(\tp) \ra(\tx)L\fr}{\HH} & \leq R \Xi^{\varrho} \norm{\varrho(\tx+\la I) \ra(\tx)}{\lh} \leq 2R \Xi^{\varrho} \varrho(\la).
\end{align}
\item[$I_{3}$]
For the last summand we argue
\begin{align}\label{err3}
&\norm{\varrho(\tp) \ga(\tx)\bx^*(\sx A\fp-\yy)}{\HH} \\ \nonumber
\leq & \Xi^{\frac{1}{2}}\Xi^{\varrho} \Psi\norm{ \ga(\tx)\varrho(\tx+\la I)(\tx+\la I)^{\frac{1}{2}}}{\lh} \\ \nonumber
\leq & \Xi^{\frac{1}{2}}\Xi^{\varrho}\Psi\tsup \varrho(t+\la)(t+\la)^{\frac{1}{2}} \abs{\ga(t)}\\ \nonumber
\leq & \Xi^{\frac{1}{2}}\Xi^{\varrho} \Psi\paren{\tsup \varrho(t+\la)(t+\la)^{-\frac{1}{2}}}\brac{\la\tsup \abs{\ga(t)}+\tsup \abs{t \ga(t)}}\\   \nonumber
\leq & \Xi^{\frac{1}{2}}\Xi^{\varrho} \Psi\brac{B+D}\varrho(\la)\la^{-\frac{1}{2}}, 
\end{align}
where~$\Xi^{1/2}$ and~$\Psi$ were as in~(\ref{ali:xizeta}) and~(\ref{ali:psi}).
\end{description}
Summarizing, using the estimates of Propositions~\ref{main.bound},~\ref{I1}, and~\eqref{over.bd2}--\eqref{err3}, we get with confidence~$1-\eta$:
\begin{equation}\label{int.bd.l}
  \norm{\fp-\fz}{\HH} \leq C\left[d(R)+\varrho(\la)\brac{R+\frac{\kappa M}{m\la}+\sqrt{\frac{\Sigma^2\NT(\la)}{m\la}}}\right]\log^4\left(\frac{4}{\eta}\right).
\end{equation}
For any parameter choice~$\la$ satisfying the condition~\eqref{l.la.condition} using the inequality~\eqref{nl} we get that
$$
\frac{\kappa M}{m\la}\leq \frac{\kappa M}{\NT(1)}
$$
and
$$
\sqrt{\frac{\Sigma^{2}\NT(\la)}{m\la}}\leq \Sigma. 
$$
This implies
\begin{equation}\label{int.bd11}
  R+\frac{\kappa
    M}{m\la}+\sqrt{\frac{\Sigma^2\NT(\la)}{m\la}}\leq 2 R,
\end{equation}
provided that~$R \geq \Sigma+\kappa M/\NT(1)$.  Inserting the bound from inequality~\eqref{int.bd11} into the estimate~(\ref{int.bd.l}) completes the proof.
\end{proof}

The bound from Theorem~\ref{err.upper.bound.gen.l} is valid for all~$R\geq \Sigma+\kappa M/\NT(1)$, and we shall now optimize the bound from Theorem~\ref{err.upper.bound.gen.l}  with respect to the choice of~$R\geq \Sigma+\kappa M/\NT(1)$. 

First, if~$\fp\in\DD(L)$ then there is~$\bar{R}\geq \Sigma+\kappa M/\NT(1)$ such that~$d(\bar{R}) = 0$, and 
$$
\norm{\fz-\fp}{\HH}\leq C
\bar{R}~\varrho(\la)\log^4\paren{\frac{4}{\eta}},
$$
where~$C$ depends on~$B$,~$D$,~$c_p$,~$\kappa$,~$n$,~$\beta$,~$\widetilde{C}$.

Otherwise, in the low smoothness case,~$\fp\not\in\DD(L)$, we introduce the following function
\begin{equation*}
  \Gamma(R) := \frac{d(R)}{R}, \qquad R\geq \Sigma+\kappa M/\NT(1),
\end{equation*}
which is non-vanishing decreasing function, and hence the inverse~$\Gamma^{-1}$ exists, and it is decreasing. Given~$\la>0$, by letting~$R = R(\lambda)$ solve the equation~$\Gamma(R) = \varrho(\la)$ we find that
\begin{equation}\label{eq:rl-bound}
  \norm{\fz-\fp}{\HH}\leq C R(\lambda)\varrho(\la)\log^4\paren{\frac{4}{\eta}},  
\end{equation}
where~$C$ depends on~$B$,~$D$,~$c_p$,~$\kappa$,~$n$,~$\beta$,~$\widetilde{C}$.

The above dependency~$\la \to R(\la)$ can be made explicit when assuming that~$\fp$ has some smoothness  measured in terms of a source condition, see Section~\ref{sec:smoothness}, below. For Theorem~\ref{err.upper.bound.gen.l} we get the error bound~\eqref{eq:rl-bound} but the parameter~$\la$ has to obey~\eqref{l.la.condition}. We will get the explicit error bound in terms of the sample size~$m$ in Corollary~\ref{cor.err.upper.bound.gen.l}.

\subsection{The regular case}\label{sec:regular}
Here we analyze the rates of convergence in the case when the underlying true solution~$\fp$ belongs to the domain of the operator~$L$. Again, we shall choose a benchmark smoothness function.

With respect to this benchmark we introduce the following distance function.

\begin{definition}[Approximate source condition]\label{app.source.cond1}
  Given~$q\geq 1$ we define the distance function~$d_{q} : [0, \infty)\to[0, \infty)$ by
  \begin{align}\label{eq.app.source1}
    d_q(R)=\inf\brac{\norm{L(f-\fp)}{\HH}:f= L^{-q}v \text{ and }\norm{v}{\HH} \leq R}. 
  \end{align}
\end{definition}
\begin{theorem}\label{err.upper.bound.gen.k}
Let~$\zz$ be i.i.d. samples drawn according to the probability measure~$\rho$. Suppose Assumptions~\ref{fp}--\ref{ass:link}. Let~$\zeta$ be any index function, such that~$\frac{1}{2}$ covers~$\zeta$. Suppose that the qualification~$p$ of the regularization~$\ga$ covers the function~$\zeta \varphi$ (for~$\varphi(t)$ from Assumption~\ref{ass:link}). Then for all~$0<\eta<1$, the following upper bound holds for the regularized solution~$\fz$ (\ref{fzl}), and for~$\la$ satisfying the condition~\eqref{l.la.condition},  with confidence~$1-\eta$: 
$$\norm{\zeta(\tp)L\paren{\fz-\fp}}{\HH}\leq C\zeta(\la)\left\{d_{q}(R)+R\paren{\varphi(\la)+\frac{1}{\sqrt{m}}}+C'\sqrt{\frac{\NT(\la)}{m\la}}\right\}\log^4\left(\frac{4}{\eta}\right),$$
Consequently, we find that 
$$\norm{\fz-\fp}{\HH} \leq C \varrho(\la)\left\{d_{q}(R)+R\paren{\varphi(\la)+\frac{1}{\sqrt{m}}}+C'\sqrt{\frac{\NT(\la)}{m\la}}\right\}\log^4\left(\frac{4}{\eta}\right)$$ 
and
$$\norm{\ip A(\fz-\fp)}{\LL} \leq C\sqrt{\la}\left\{d_{q}(R)+R\paren{\varphi(\la)+\frac{1}{\sqrt{m}}}+C'\sqrt{\frac{\NT(\la)}{m\la}}\right\}\log^4\left(\frac{4}{\eta}\right),$$  
where~$C$ depends on~$B$,~$D$,~$c_p$,~$\kappa$, and~$C'=2\kappa M+\Sigma$.  
\end{theorem}

\begin{proof}
For the minimizer~$\fr$ of the distance function defined in \eqref{eq.app.source1}, the error can be expressed as follows:
\begin{equation*}
    L(\fp-\fz)=\ra(\tx)L(\fp-\fr)+\ra(\tx)L\fr+\ga(\tx)\bx^*(\sx A \fp-\yy).
\end{equation*} 

First, we estimate the error in the interpolation norm for some index function~$\zeta$:
  \begin{align}\label{interpolation.bd}
    \norm{\zeta(\tp)L(\fp-\fz)}{\HH}\leq &
                                           d_{q}(R)\underbrace{\norm{\zeta(\tp)\ra(\tx)}{\lh}}_{I_{1}}
  +\underbrace{\norm{\zeta(\tp)\ra(\tx)L\fr}{\HH}}_{I_{2}}\\  \notag
  &+\underbrace{\norm{\zeta(\tp) \ga(\tx)\bx^*(\sx A\fp-\yy)}{\HH}}_{I_{3}}.
  \end{align} 

  \begin{description}
  \item[$I_1$]
We bound,
  \begin{align}\label{over.bound}
\norm{\zeta(\tp) \ra(\tx)}{\lh} \leq \norm{\zeta(\tp+\la I) \ra(\tx)}{\lh} 
    \leq \Xi^\zeta\norm{\zeta(\tx+\la I) \ra(\tx)}{\lh} 
    \leq \Xi^\zeta c_p\zeta(\la).
  \end{align}
\item[$I_{2}$]
For the minimizer~$\fr = L^{-q}g$ of the distance function~(\ref{eq.app.source1}), we observe from Proposition~\ref{prop:Heinz} that there is~$v\in\HH$ such that~$L\fr=L^{-(q-1)}g = \varphi(\tp)v$,~$\norm{v}{\HH}\leq R$. Thus by assuming that the function~$\varphi= \varphi_{1} \varphi_{2}$ with~$\varphi_{1}$ being sub-linear and~$\varphi_{2}$ Lipschitz (with constant one) we continue bounding 
\begin{align*}
\ra(\tx)L\fr=\ra(\tx)\varphi(\tp)v = \ra(\tx)\varphi_2(\tx)\varphi_1(\tp)v+\ra(\tx)(\varphi_2(\tp)-\varphi_2(\tx))\varphi_1(\tp)v.
\end{align*}

Then we get,
\begin{align}
&\norm{\zeta(\tp)\ra(\tx)L\fr}{\HH}=\norm{\zeta(\tp)\ra(\tx)\varphi(\tp)v}{\HH}\\  \nonumber
\leq & \Xi^\zeta\brac{\norm{\zeta(\tx+\la I)\ra(\tx)\varphi_2(\tx)\varphi_1(\tp)v}{\HH}+\norm{\zeta(\tx+\la I)\ra(\tx)(\varphi_2(\tp)-\varphi_2(\tx))\varphi_1(\tp)v}{\HH}}\\  \nonumber
\leq & R \Xi^\zeta\left\lbrace \norm{\zeta(\tx+\la I)\ra(\tx)\varphi_2(\tx)\varphi_1(\tx+\la I)}{\lh}\norm{\paren{\frac{1}{\varphi_1}}(\tx+\la I)\varphi_1(\tp+\la I)}{\lh}\right. \\ \nonumber
&\left.+\varphi_1(\kappa^2)\norm{\zeta(\tx+\la I)\ra(\tx)}{\lh}\norm{\tp-\tx}{\lh} \right\rbrace \\  \nonumber
\leq&  R\Xi^\zeta\left\{\Xi^{\varphi_1}\tsup\brac{\abs{\ra(t)}\varphi_2(t)\zeta(t+\la)\varphi_1(t+\la)}\right. \\ \nonumber
&\left.+\varphi_1(\kappa^2)\norm{\tp-\tx}{\lh}\tsup\brac{\abs{\ra(t)}\zeta(t+\la)}\right\}\\  \nonumber
\leq & R 2^qc_p\zeta(\la)\Xi^\zeta\brac{\Xi^{\varphi_1}\varphi(\la)+\varphi_1(\kappa^2)\norm{\tp-\tx}{\lh}},
\end{align}
because of the qualification of the regularization. 
\item[$I_{3}$]
From the arguments used in \eqref{err3}, we get
\begin{align}\label{err1}
\norm{\zeta(\tp) \ga(\tx)\bx^*(\sx A \fp-\yy)}{\HH} \leq \Xi^{\frac{1}{2}}\Xi^{\zeta} \Psi\brac{B+D}\zeta(\la)\la^{-\frac{1}{2}}.
\end{align}
  \end{description} 
Overall, using Propositions~\ref{main.bound}--\ref{I1} and \eqref{over.bound}--\eqref{err1} in \eqref{interpolation.bd} we obtain with confidence~$1-\eta$:
\begin{equation}\label{zeta.bd}
\norm{\zeta(\tp)L\paren{\fz-\fp}}{\HH}\leq C\zeta(\la)\left\{d_{q}(R)+R\paren{\varphi(\la)+\frac{1}{\sqrt{m}}}+\frac{\kappa M}{m\la}+\sqrt{\frac{\Sigma^2\NT(\la)}{m\la}}\right\}\log^4\left(\frac{4}{\eta}\right).
\end{equation}
The fact that~$\NT(\la)$ is decreasing function of~$\la$ with the inequality~\eqref{l.la.condition} implies that
$$
\frac{\kappa M}{m\la}\leq \frac{\kappa M}{m\la}\frac{\NT(\la)}{\NT(1)}\leq \frac{\kappa M}{\NT(1)}\sqrt{\frac{\NT(\la)}{m\la}}. 
$$
This, together with~\eqref{zeta.bd} yields the first result.

For the last two estimates in Theorem~\ref{err.upper.bound.gen.k}, by using Proposition~\ref{prop:Heinz} we get,
\begin{align*}\label{norm.bd}
\norm{\fp-\fz}{\HH} & = \norm{L^{-1}\brac{L(\fp-\fz)}}{\HH}\leq\norm{\varrho(\tp)L(\fp-\fz)}{\HH},
\end{align*}
and
\begin{equation*}\label{direct.norm}
\norm{\ip A(\fz-\fp)}{\LL}=\norm{\tp^{1/2}L(\fz-\fp)}{\HH}.
\end{equation*}
These two upper bounds can now be estimated from the general bound by letting~$\zeta:=\varrho$ and~$\zeta(t):=t^\frac{1}{2}$, respectively. We also use that~$\varrho^2$ is sub-linear, and this completes the proof.
\end{proof}

The bound from Theorem~\ref{err.upper.bound.gen.k} is valid for all~$R\geq 1$, and we shall now optimize the bound from Theorem~\ref{err.upper.bound.gen.k} with respect to the choice of~$R\geq 1$.

First, if~$\fp\in\range\paren{L^{-q}}$ then~$d_{q}(\bar{R}) = 0$ for some~$\bar{R}$, we find that

$$\norm{\fz-\fp}{\HH}\leq C\varrho\paren{\la}\left\{\bar{R}\paren{\varphi(\la)+\frac{1}{\sqrt{m}}}+C'\sqrt{\frac{\NT(\la)}{m\la}}\right\}\log^4\paren{\frac{4}{\eta}}.$$

Otherwise, in case that~$\fp\not\in\range\paren{L^{-q}}$ we introduce the following function
\begin{equation}
  \Gamma_{q}(R) := \frac{d_{q}(R)}{R}, \qquad R \geq 1,
\end{equation}
which is non-vanishing decreasing function, and hence the inverse~$\Gamma_{q}^{-1}$ exists and it is decreasing. We finally get the main result, by letting~$R = R(\la)$ solving the equation~$\Gamma_{q}(R) = \varphi(\la)$, and we find that

$$\norm{\fz-\fp}{\HH}\leq C\varrho\paren{\la}\brac{R(\la)\paren{\varphi(\la)+\frac{1}{\sqrt{m}}}+C'\sqrt{\frac{\NT(\la)}{m\la}}}\log^4\paren{\frac{4}{\eta}}.$$

\section{Smoothness in terms of source-wise representation}\label{sec:smoothness}

Here we shall specify the smoothness of the true solution in terms of the bounded linear injection and self-adjoint operator~$L^{-1}$. 

\begin{assumption}[General source condition]\label{source.cond}
For an index function~$\theta$, the true solution~$\fp$ belongs to the class~$\Omega(\theta,R^\dagger)$ with
  \begin{equation*}
    \Omega(\theta,R^\dagger):=\left\{f \in \HH: f= \theta(L^{-1})v \text{ and }\norm{v}{\HH} \leq R^\dagger\right\}.
  \end{equation*}
\end{assumption}
In the special case when the function~$\theta(t) := t^{r}$ is a power function, such source-wise representation is called H{\"o}lder type.

We aim at bounding the distance functions~$d(R)$ and~$d_q(R)$ from the oversmoothing and regular cases, respectively.

\subsection{The oversmoothing case}\label{Sec:obersmoothing-theta}
Here the benchmark source condition is linear, and we shall thus assume that the index function~$\theta$ is sub-linear. The obtained bounds will rely on the results from~\cite[Theorem~5.9]{Hofmann2007}. We denote the identity function~$\iota:t \mapsto t$, representing the benchmark smoothness index function. Under Assumption~\ref{source.cond} we find that
$$
d(R) \leq R \paren{\paren{\frac{\iota}{\theta}}^{-1}\paren{\frac{R^{\dag}}   R}},\quad R>0.
$$
In order to minimize the bound from Theorem~\ref{err.upper.bound.gen.l}, we balance~$d(R) = R \varrho(\la)$, resulting in
\begin{equation}\label{R.choice1}
  R(\la) = R^{\dag}\frac{\theta\paren{\varrho(\la)}}{\varrho(\la)},\quad \la>0.
\end{equation}

Thus, for this value of~$R(\la)$ under the condition~(\ref{l.la.condition}), the bound~(\ref{eq:rl-bound}) reduces to
\begin{equation}
  \label{eq:thm3.3-reduced}
  \norm{\fz - \fp}{\HH} \leq C R(\la) \varrho(\la) \log^4(4/\eta)\leq
  C R^{\dag}\theta\paren{\varrho(\la)}\log^4(4/\eta).
\end{equation}

The following corollary is the consequence of Theorem~\ref{err.upper.bound.gen.l} which explicitly provide the error bound under the parameter choice of~$\la$ in terms of the sample size~$m$.

\begin{corollary}\label{cor.err.upper.bound.gen.l}
Under the same assumptions of Theorem~\ref{err.upper.bound.gen.l} and Assumption~\ref{source.cond} for the sub-linear function~$\theta$ with the a-priori choice of the regularization parameter~$\la^{\ast}$ from solving the equation~$\mathcal N_{\tp}(\la^{\ast}) = m \la^{\ast}$, for all~$0<\eta<1$, the following error estimates holds with confidence~$1-\eta$:
$$\norm{\fz-\fp}{\HH} \leq  C \theta\paren{\varrho\paren{\la^*}}\log^4\left(\frac{4}{\eta}\right),$$
where~$C$ depends on~$B$,~$D$,~$c_p$,~$\kappa$,~$n$,~$\beta$,~$\widetilde{C}$,~$M$,~$\Sigma$, and~$R^\dagger$.
\end{corollary}
We observe that the above parameter choice evidently satisfies condition~\eqref{l.la.condition}.

\subsection{The regular case}\label{sec:regular-theta}
In this case the benchmark is given by the index function~$\iota^{q}$, and we shall assume that the given smoothness, measured in terms of~$\theta$, is such that the function~$\iota^{q}/\theta$ for~$0 < t \leq \kappa^2$, is an index function. However, the definition of the distance function~$R \mapsto d_{q}(R)$ is non-standard. The target norm is~$\norm{L(f - \fp)}{\HH}$, and, in order to apply the result from~\cite[Theorem~5.9]{Hofmann2007} we have to `rescale' the given smoothness (in terms of the operator~$L^{-1}$) by factor~$L^{-1}$. If Assumption~\ref{source.cond} holds true with index function~$\theta$, for which the quotient~$\iota^{q}/\theta$ is an index function, and so will be the function~$\iota^{q-1}/(\theta/\iota)$, then this results in the bound
\begin{equation}\label{R.choice}
  d_{q}(R) \leq R~\sbrac{\paren{\frac{\iota^{q}}{\theta}}^{-1}\paren{\frac{R^\dagger}{R}}}^{q-1},\quad R>0.
\end{equation}

According to Theorem~\ref{err.upper.bound.gen.k} we balance
$$
d_{q}(R) = R \varphi(\la).
$$
This yields
$$
R(\la) = R^{\dag} \frac{\theta\paren{\varrho(\la)}}{\varrho^{q}(\la)},\quad R>0.
$$

Inserting this bound into Theorem~\ref{err.upper.bound.gen.k} we find that
\begin{align}\label{eqn:reg.bd}
  \norm{\fz-\fp}{\HH} &\leq C\varrho(\la)\left\{R^\dagger\frac{\theta\paren{\varrho(\la)}}{\varrho(\la)}\paren{1+\frac{1}{\sqrt{m}\varphi(\la)}}+C'\sqrt{\frac{\NT(\la)}{m\la}}\right\}\log^4\paren{\frac{4}{\eta}}\\ \nonumber
                      & = C\varrho(\la)\left\{R^\dagger\frac{\theta\paren{\varrho(\la)}}{\varrho(\la)}+\frac{1}{\sqrt{m}}\paren{R^\dagger\frac{\theta\paren{\varrho(\la)}}{\varrho^q(\la)}+C'\sqrt{\frac{\NT(\la)}{\la}}}\right\}\log^4\paren{\frac{4}{\eta}}
\end{align}
provided that \eqref{l.la.condition} holds.

The optimization of the bound in the inequality~\eqref{eqn:reg.bd} depends on which term is dominant in the last two summands. Then we can balance the remaining (two) terms. This results in the following corollaries for the different choices of the regularization parameter:

\begin{corollary}\label{coro.err.upper.bound.gen.k1}
Suppose~$\frac{\iota^q}{\theta}(t)$ and~$\frac{\iota^q}{\theta}\paren{\varrho(t)}\sqrt{\frac{\NT(t)}{t}}$ are the index functions. Then under the same assumptions of Theorem~\ref{err.upper.bound.gen.k} and Assumption~\ref{source.cond} with the a-priori choice of the regularization parameter~$\la^*=\varphi^{-1}\paren{\frac{1}{\sqrt{m}}}$, for all~$0<\eta<1$, the following upper bound holds with confidence~$1-\eta$:

$$\norm{\fz-\fp}{\HH}\leq C\theta\paren{\varrho(\la^*)}\log^4\paren{\frac{4}{\eta}},$$
where~$C$ depends on~$B$,~$D$,~$c_p$,~$\kappa$,~$M$,~$\Sigma$, and~$R^\dagger$.
\end{corollary}

\begin{corollary}\label{coro.err.upper.bound.gen.k}
Suppose~$\frac{\iota^q}{\theta}(t)$ and~$\frac{\theta}{\iota^q}\paren{\varrho(t)}\sqrt{\frac{t}{\NT(t)}}$ are the index functions. Then under the same assumptions of Theorem~\ref{err.upper.bound.gen.k} and Assumption~\ref{source.cond} with the a-priori choice of the regularization parameter~$\la^*$ as solution to the equation~$\frac{\Theta^{2}(\varrho(\la^{\ast}))}{\varrho^{2}(\la^{\ast})} \la^{\ast}m = \mathcal N_{\tp}(\la^{\ast})$, for all~$0<\eta<1$, the following upper bound holds with confidence~$1-\eta$:
$$
\norm{\fz-\fp}{\HH}\leq C\theta\paren{\varrho(\la^*)}\log^4\paren{\frac{4}{\eta}},
$$
where~$C$ depends on~$B$,~$D$,~$c_p$,~$\kappa$,~$M$,~$\Sigma$, and~$R^\dagger$.
\end{corollary}
Since by assumption the function~$t \mapsto \frac{\Theta^{2}(\varrho(\la^{\ast}))}{\varrho^{2}(\la^{\ast})}$ is an index function we will have that condition~(\ref{l.la.condition}) holds for~$m$ large enough.

\subsection{Taking the behavior of effective dimension into account}\label{sec:taking-eff-dim}

Below, to be specific, we consider the following two behaviors of the decay of the effective dimensions, say power-type and logarithm type, which is known to hold true in many situations.
\begin{assumption}[Polynomial decay condition]\label{poly.decay}
Assume that there exists some positive constant~$c>0$ such that
\begin{equation*}
\NT(\la) \leq c\la^{-b},\quad\text{ for }0\leq b<1,~\forall \la>0.
\end{equation*}
\end{assumption}

\begin{assumption}[Logarithmic decay condition]\label{log.decay}
Assume that there exists some positive constant~$c>0$ such that
\begin{equation*}
\NT(\la)\leq c\log\left(\frac{1}{\la}\right), \quad\forall \la>0.
\end{equation*}
\end{assumption}

\begin{remark}
We mention that a polynomial decay of the eigenvalues of the covariance operator~$\tp$ yields the polynomial-type behavior of the effective dimension, see~\cite{Caponnetto}. Rather in some situations this behavior is not evident. Lu et al. \cite{Lu2020} showed that for Gaussian kernel~$K_1(x,x') = xx' + e^{-8(x-x')^2}$ with the uniform sampling on~$[0,1]$, the effective dimension exhibits the log-type behavior (Assumption~\ref{log.decay}), on the other hand, the kernel~$K_2(x,x') = \min\{x,x'\}-xt$ exhibits the power-type behavior (Assumption~\ref{poly.decay}).
\end{remark}

\begin{table}[ht!]
  \centering {\renewcommand{\arraystretch}{2}
  \caption{Convergence rates of the regularized solution~$\fz$ for~$a\leq \frac{1}{2}$,~$a q\leq {p}$ under Assumption~\ref{poly.decay}.}\label{comparision.1}
    \begin{tabular}{|l|l|l|l|l|l|}
      \hline
      \multirow{1.5}{4em}{\small{Case}}  &\small{Convergence}  & \small{Parameter} &  \small{True} & \small{Benchmark}  & \multirow{1.5}{4em}{\small{Conditions}} \\  [-15pt]
                                         & \small{rates}   &~$\la^* = \mathcal O(\cdot)$ &  \small{Smoothness} & \small{Smoothness}  &  \\
      \hline
      \small{{\bf Oversmoothing}}  &~$\paren{\frac{1}{\sqrt{m}}}^{\frac{2ar}{b+1}}$ &~$\paren{\frac{1}{\sqrt{m}}}^{\frac{2}{b+1}}$ &~$r\leq 1$ &~$q=1$ &~$a\geq \frac{1}{n+1}$
      \\
      \hline
      \multirow{2}{4em}{\small{{\bf Regular}}} &~$\paren{\frac{1}{\sqrt{m}}}^{\frac{r}{q-1}}$ &
                                                                                          ~$\paren{\frac{1}{\sqrt{m}}}^{\frac{1}{a(q-1)}}$ & \multirow{2}{4em}{$r\geq 1$}  & \multirow{2}{4em}{$q >1$} &~$a q\geq a r+\frac{b+1}{2}$ \\ 
                                         &~$\paren{\frac{1}{\sqrt{m}}}^{\frac{2ar}{2ar+b+1-2a}}$ &~$\paren{\frac{1}{\sqrt{m}}}^{\frac{2}{2ar+b+1-2a}}$ &  &  &~$a r \leq a q\leq a r+\frac{b+1}{2}$\\
      \hline
    \end{tabular}
  }
\end{table}
%%%%%%%%%%%%%%%%%%%%%%%%%%%%%%%%%%%%%%%%%%%%%%%%%%%%%%%%%%%%%%%%%%%%%%%
\begin{table}[ht!]
  \centering {\renewcommand{\arraystretch}{2}
  \caption{Convergence rates of the regularized solution~$\fz$ for~$a\leq \frac{1}{2}$,~$a q\leq {p}$ under Assumption~\ref{log.decay}.}\label{comparision.2}
    \begin{tabular}{|l|l|l|l|l|l|}
      \hline
      \multirow{1.5}{4em}{\small{Case}} &\small{Convergence}  & \small{Parameter} &  \small{True} & \small{Benchmark}  & \multirow{1.5}{4em}{\small{Conditions}} \\  [-15pt]
                                        & \small{rates}   &~$\la^* = \mathcal O(\cdot)$ &  \small{Smoothness} & \small{Smoothness}  &  \\
      \hline
      \small{{\bf Oversmoothing}}  &~$\paren{\frac{\log  m}{m}}^{\frac{ar}{b+1}}$ &~$\paren{\frac{\log  m}{m}}^{\frac{1}{b+1}}$ &~$r\leq 1$ &~$q=1$ &~$a\geq \frac{1}{n+1}$
      \\
      \hline
      \multirow{2}{4em}{\small{{\bf Regular}}} &~$\paren{\frac{\log  m}{m}}^{\frac{r}{2(q-1)}}$ &
                                                                                            ~$\paren{\frac{\log m}{m}}^{\frac{1}{2a(q-1)}}$ & \multirow{2}{4em}{$r\geq 1$} & \multirow{2}{4em}{$q >1$} &~$a q\geq a r+\frac{b+1}{2}$ \\ 
                                        &~$\paren{\frac{\log  m}{m}}^{\frac{ar}{2ar+b+1-2a}}$ &~$\paren{\frac{\log m}{m}}^{\frac{1}{2ar+b+1-2a}}$ &  &  &~$a r \leq a q\leq a r+\frac{b+1}{2}$\\
      \hline
    \end{tabular}
  }
\end{table}

In Tables~\ref{comparision.1} and~\ref{comparision.2} we present the convergence rates under the specific behavior of the effective dimension (Assumptions~\ref{poly.decay} and~\ref{log.decay}, respectively). For a clear picture of the error analysis, we present the error bounds in the particular case when the link condition as well as the source condition are of power type, i.e.,\ $\varrho(t)=t^a$ and~$\theta(t)=t^r$ for parameters~$a,r>0$. The qualification of the regularization is denoted by~$p$ as before. Also, the benchmark smoothness is~$q$, where either~$q=1$ (oversmoothning case) or~$q>1$ (regular case).  Notice, that due to the sub-linearity condition for~$\varrho^{2}$ we must have that~$0 < a \leq 1/2$. Also, throughout the analysis, we assume that the qualification covers the given smoothness, i.e.,\ $a q\leq p$. The bounds presented in the tables are consequences of Corollaries~\ref{cor.err.upper.bound.gen.l}--\ref{coro.err.upper.bound.gen.k}, respectively. Therefore Assumptions~\ref{fp}--\ref{source.cond} are assumed to be satisfied for the following results.

The table is structured as follows. In the first column we present the rates of convergence~$\varepsilon(m)$ for the error estimates of the form:
$$
\mathbb{P}_{\zz\in Z^m}\brac{\norm{\fz-\fp}{\HH}\leq  C\varepsilon(m)\log^4\paren{\frac{4}{\eta}}}\geq 1-\eta.
$$

In the second column, the corresponding order of the regularization parameter choice~$\la^*$ in terms of~$m$ is indicated. In the third and fourth columns, we highlight the smoothness of the true solution~$\fp$, and the benchmark smoothness, respectively. The fifth column presents the parameter involved in the link condition. In the last column, we emphasize additional constraints, specifically on the benchmark smoothness.

The first row corresponds to the oversmoothing case, and the last two rows correspond to the regular case. In the regular case, we observe that the validity of the rates of the convergence depends on the benchmark smoothness through~$a q$.  At the intersection point, when~$a q = a r+\frac{b+1}{2}$, then both rates coincide. As we will see in the next section the rates of convergence in the regular case ($q>1$) are optimal provided that the benchmark smoothness is chosen appropriately.
\section{Optimality of the error bounds}\label{sec:optimality}
We shall discuss the optimality of the previously obtained error bounds, in the regular case, and we shall use the known optimality results from~\cite{Blanchard}. However, at present the smoothness is measured with respect to the operator~$\tp$, whereas in~\cite{Blanchard} this was done with respect to the operator~$\lp:= A^{ \ast} \ip^{\ast} \ip A = L\tp L$. Therefore, the following `recipe' will be used.
\begin{enumerate}
\item Transfer smoothness as given in terms of~$L^{-1}$ to smoothness in terms of~$\lp$, and
\item Knowing the decay of the singular numbers of the operator~$\tp$ inherent in Assumption~\ref{poly.decay}, find the decay of the singular numbers of~$\lp$.
\end{enumerate}
In order to keep the analysis simple and transparent we confine to power type smoothness~$\theta(t)=t^{r},\ 0 < r \leq q$ in Assumption~\ref{source.cond}, as well as to power type link in Assumption~\ref{ass:link} with~$\varrho(t) := t^{a}$ for some~$a>0$. 

\subsection{Relating smoothness}\label{sec:smooth-relate}

The link condition is crucial, and the subsequent arguments are of interpolation type, applying Heinz Inequality within the present context. To this end, we require that~$q$ is chosen such that~$aq\geq 1/2$. In this case Assumption~\ref{ass:link} yields, by applying Heinz Inequality with exponent~$1/(2aq)\leq 1$ that
$$
\norm{\ip A L^{-1}u}{\LL} = \norm{\tp^{1/2}u}{\HH} \asymp\footnote{We shall suppress the recalculations
of the corresponding constants.} \norm{L^{-\frac 1 {2a}}u}{\HH},\quad u\in\HH.
$$
Letting~$v:= L^{-1}u$ we find that
\begin{equation}
  \label{eq:snul}
\norm{\lp^{1/2} v}{\HH} = \norm{\ip A v}{\LL} \asymp  \norm{L^{-(\frac 1 {2a} -1)}v}{\HH},\quad v\in\HH.
\end{equation}
First, we see from this that~$a < 1/2$, because otherwise~$\lp$ would be continuously invertible. Also, the relation~(\ref{eq:snul}) would allow transferring smoothness~$r$ with respect to~$L^{-1}$ to~$\lp$ as long as~$0 < r \leq \frac 1 {2a} - 1$. In order to treat higher smoothness (in terms of~$L^{-1}$) a lifting condition is unavoidable. This must be consistent with the link from~(\ref{eq:snul}). Thus we look for a factor~$z$ such that~$t^{(\frac 1 {2a} -1)z} = t^{q}$, yielding~$z:= \frac{2aq}{1 - 2a}$.
\begin{assumption}[lifting condition]\label{ass:lift}
  We have that
$$
 \norm{L^{-q}u}{\HH} \asymp \norm{\lp^{\frac{aq}{1-2a}}u}{\HH},\quad u\in\HH.
$$  
\end{assumption}
Having this lifting, and applying Heinz Inequality (with exponent~$r/q$) yields
\begin{equation}
  \label{eq:smoothsnu}
\norm{L^{-r}v}{\HH} \asymp \norm{\lp^{\frac{a r}{1-2a}}v}{\HH},\quad v\in\HH,  
\end{equation}
and a source-wise representation as in Assumption~\ref{source.cond} yields a corresponding source-wise representation with respect to the operator~$\lp$ (with different constant).

\subsection{Relating effective dimensions}\label{sec:effdim-relate}

Here we shall use the following consequence of the link condition in Assumption~\ref{ass:link}. Indeed, by squaring the norms we see that
$$
\scalar{L^{-2q}u}{u} \asymp \scalar{\tp^{2aq}u}{u},\quad u\in\HH.
$$
The Weyl Monotonicity Theorem~\cite[Cor.~III.2.3]{Bhatia1997} yields that then~$s_{j}(L^{-2q}) \asymp s_{j}(\tp^{2aq}),\ j=1,2,\dots$, or simplified that~$s_{j}(L^{-1}) \asymp s_{j}^{a}(\tp),\ j=1,2,\dots$ by spectral calculus. Here~$s_{j}(L^{-1})$ and~$s_{j}(\tp)$ denote the singular numbers of the operators. Similarly, we obtain from~(\ref{eq:snul}) that~$s_{j}(\lp) \asymp s_{j}^{\frac{1 - 2a}{a}}(L^{-1})$, and a fortiori that~$s_{j}(\lp) \asymp s_{j}^{1 - 2a}(\tp)$.

\subsection{Lower bound}\label{sec:lowerb}

In order to show the optimality of the error bounds  as discussed in Table~\ref{comparision.1}, we shall assure that the decay of the effective dimension cannot be faster than asserted in Assumption~\ref{poly.decay}.
\begin{assumption}
There is a constant~$c>0$ such that the singular numbers of the operator~$\tp$ obey
$$
  s_{j}(\tp) \geq c j^{-1/b},\quad j=1,2,\dots
$$
\end{assumption}
Notice that this yields that~$\mathcal N(\la) \geq c \la^{-b}$, such that this is the limiting case for which Assumption~\ref{poly.decay} holds. The following is reported in~\cite{Blanchard} for the problem~(\ref{Model}): Under smoothness~$r$ with respect to the operator~$\lp$, and with the decay of the singular numbers~$s_{j}(\lp)$ not faster than~$j^{-1/b}$, the optimal rate is of the order~$\paren{\frac{1}{\sqrt{m}}}^{\frac{2r}{2r + b +1}}$.  In the present context, we have to assign~$r\leftarrow \frac{ar}{1-2a}$ and~$b\leftarrow \frac{b}{1-2a}$. This yield a lower bound of the order
$$
\paren{\frac{1}{\sqrt{m}}}^{\frac{2ar/(1 - 2a)}{2ar/(1 - 2a) + b/(1 - 2a) + 1}} = \paren{\frac{1}{\sqrt{m}}}^{\frac{2ar}{2ar + b + 1 - 2a}}.
$$
This corresponds to the upper bound as discussed in the last row of Table~\ref{comparision.1}, and it shows that the rate is of optimal order.

\section{Conclusion}\label{Sec:conclusion}
We summarize the above findings. We investigated regularization in Hilbert scales for the considered inverse problem with general centered noise, which is assumed to obey a Bernstein-type moment condition. This noise condition is not required when the output space is bounded. We analyzed regularization in a Hilbert scale, generated by some unbounded operator~$L$. In order to do so we used a link condition to transfer information from~$L^{-1}$ to~$\tp$, the underlying covariance operator. 

In the main body, we established error bounds in terms of distance functions, which measure the deviation of the regression function to some benchmark smoothness. These error bounds were then specified for smoothness given in terms of solution smoothness with respect to the operator~$L^{-1}$, by bounding the corresponding distance functions. The error estimates are explicitly described as the exponential deviation inequalities in terms of the sample size which holds non-asymptotically in the probabilistic sense. We discussed the convergence rates for both oversmoothing and regular cases under different behavior of the effective dimension in reproducing kernel approach. In particular, optimal convergence rates can be achieved with the appropriate choice of benchmark smoothness and an a-priori parameter choice for the regular case. Although we mainly focused bounding the reconstruction error~$\norm{\fz-\fp}{\HH}$, error estimates of the prediction error~$\norm{\ip A(\fz-\fp)}{\LL}$ can also be derived similarly in terms of sample size using Theorem~\ref{err.upper.bound.gen.k}.  The optimal parameter choice depends on the unknown parameters~$a$,~$b$,~$r$, reflecting the link condition, the decay of the effective dimension, and the solution smoothness. Therefore a data-driven parameter choice may be required to apply the regularization algorithms. This will be a topic of future research. 
 
\appendix

\section{Proof of Proposition~\ref{prop:relation.eff_dim}}\label{sec:lemma-proof}

We start with the following technical result.

\begin{lemma}\label{Lem:tplp}
Suppose that the function~$\varrho$ from the link condition is such that the function~$t\mapsto \paren{\varrho^{2q}}^{-1}(t)$ is operator concave, and that there is some~$n\in\NN$ for which the function~$t\mapsto \varrho^{-1}(t)/t^n$ is concave.  Under Assumption~\ref{ass:link} we have that
\begin{equation*}
\frac{s_j\paren{\tp}}{s_j\paren{\varrho^2(\tp)}} \leq \beta^{n-1} s_j\paren{\lp } \leq \beta^{2n}\frac{s_j\paren{\tp}}{s_j\paren{\varrho^2(\tp)}},\quad j=1,2,\dots
\end{equation*}
\end{lemma}
\begin{proof}
The proof is based on two consequences of Assumption~\ref{ass:link}, which, in terms of the partial ordering for self-adjoint operators in Hilbert space can be restated as  
\begin{equation*}
\scalar{(L^{-1})^{2q}u}{u}{\HH} \leq \scalar{\varrho^{2q}(\tp)u}{u}{\HH} \leq\scalar{(\beta L^{-1})^{2q}u}{u}{\HH},\quad u\in\HH. 
\end{equation*}
Applying the operator concave function~$t\mapsto \paren{\varrho^{2q}}^{-1}(t)$ respects the partial ordering, and we obtain\footnote{we use that~$\paren{\varrho^{2q}}^{-1}(t^{2q}) =\paren{\varrho}^{-1}(t)$.} that
\begin{equation*}
\scalar{\varrho^{-1}(L^{-1})u}{u}{\HH} \leq \scalar{\tp u}{u}{\HH} \leq\scalar{\varrho^{-1}(\beta L^{-1})u}{u}{\HH}.
\end{equation*}
Letting~$u:= Lv\in\HH$, and since by construction~$\tp = L^{-1}\lp L^{-1}$ we deduce that
\begin{equation*}
\scalar{\varrho^{-1}(L^{-1})L^2 v}{v}{\HH} \leq \scalar{\lp v}{v}{\HH} \leq \scalar{\varrho^{-1}(\beta L^{-1})L^2 v}{v}{\HH},\quad v\in\DD(L).
\end{equation*}
The sub-linearity of~$\varrho^{2}$ implies that the function~$t\mapsto \varrho^{-1}(t)/t^{2}$ is non-decreasing, such that the operator~$\varrho^{-1}(\beta L^{-1})L^2$ is bounded, and hence the above inequality extends to~$v\in\HH$. Next we apply the Weyl Monotonicity Theorem~\cite[Cor.~III.2.3]{Bhatia1997} to see that
\begin{equation}\label{eq.sing1}
\frac{s_j\paren{\varrho^{-1}(L^{-1})}}{s_j^2(L^{-1})} \leq s_j\paren{\lp } \leq \frac{s_j\paren{\varrho^{-1}(\beta L^{-1})}}{s_j^2(L^{-1})},\quad j=1,2,\dots 
\end{equation}
Applying this theorem to the first inequality in Proposition~\ref{prop:Heinz} we also find that
\begin{equation*}
s_j\paren{\varrho^{-1}(L^{-1})} \leq s_j\paren{\tp } \leq s_j\paren{\varrho^{-1}(\beta L^{-1})},\ j=1,2,\dots 
\end{equation*}
To proceed we shall use that the sub-linearity of the function~$ \varrho^2$ and the concavity of the function~$\varsigma(t):=\varrho^{-1}(t)/t^n$. This yields that~$\varsigma(\beta t)\leq \beta\varsigma(t),\ \beta\geq 1$ and overall, we find that
 \begin{equation*}
 \frac{s_j\paren{\varrho^{-1}(L^{-1})}}{s_j^2(L^{-1})} \leq \frac{s_j\paren{\tp}}{s_j\paren{\varrho^2(\tp)}} \leq \frac{s_j\paren{\varrho^{-1}(\beta L^{-1})}}{s_j^2(\beta L^{-1})}  = \beta^{n-2} s_j^{n-2}(L^{-1})\frac{s_j\paren{\varrho^{-1}(\beta L^{-1})}}{s_j^n(\beta L^{-1})} \leq \beta^{n-1}\frac{s_j\paren{\varrho^{-1}(L^{-1})}}{s_j^2(L^{-1})}.
 \end{equation*}
This, together with the inequalities~(\ref{eq.sing1}) gives
\begin{equation*}
\frac{s_j\paren{\tp}}{s_j\paren{\varrho^2(\tp)}} \leq \beta^{n-1} s_j\paren{\lp } \leq \beta^{2n}\frac{s_j\paren{\tp}}{s_j\paren{\varrho^2(\tp)}},
\end{equation*}
and the proof is complete.
\end{proof}

\begin{proof}[Proof of Proposition~\ref{prop:relation.eff_dim}]
Since the function~$t\mapsto t/\varrho^{2}(t)$ is assumed to be an index function, we find from Lemma~\ref{Lem:tplp} that the implication
$$
\beta^{n+1}\frac{\la}{\varrho^2(\la)} \leq s_{j}(\lp) \quad \text{implies} \quad \la \leq s_{j}(\tp) 
$$
holds true. This yields
\begin{equation}\label{tplp.relation}
\#\brac{j, \quad s_{j}(\lp) \geq \beta^{n+1}\frac{\la}{\varrho^2(\la)}} \leq \#\brac{j,\quad s_{j}(\tp) \geq \la},\quad \la \leq \norm{\tp}{\mathcal L(\HH)}.
\end{equation}

As a consequence of~\cite[Prop.~6]{Lin2015} there is~$\widetilde C$ such that
\begin{equation*}
\NL(\la)\leq \widetilde{C}\#\brac{j, \quad s_j(\lp)\geq \la}.
\end{equation*}
This, together with~\eqref{tplp.relation}, implies that
\begin{align*}
\NL\paren{\beta^{n+1}\frac{\la}{\varrho^2(\la)}} \leq & \widetilde{C}\#\brac{j, \quad s_j(\lp)\geq \beta^{n+1}\frac{\la}{\varrho^2(\la)}}\leq \widetilde{C}\#\brac{j, \quad s_j(\tp)\geq \la} \\
= & 2\widetilde{C}\sum\limits_{s_j(\tp)\geq \la}\frac{1}{2} \leq 2\widetilde{C}\sum\limits_{j=1}^\infty \frac{s_j(\tp)}{\la+s_j(\tp)} = 2\widetilde{C}\NT(\la),\quad \la \leq \norm{\tp}{\mathcal L(\HH)}.
\end{align*}
Since the function~$\la\mapsto \la \NL(\la)~$ is non-decreasing we continue to bound
\begin{equation}\label{relation.eff_dim}
 \NL\paren{\frac{\la}{\varrho^2(\la)}} \leq \beta^{n+1}\NL\paren{\beta^{n+1}\frac{\la}{\varrho^2(\la)}} \leq 2\beta^{n+1}\widetilde{C}\NT(\la),\quad \la \leq \norm{\tp}{\mathcal L(\HH)},
\end{equation}
which completes the proof.
\end{proof}

\section{Probabilistic bounds}\label{Sec:Prob.bound}

In the following proposition, we present the standard perturbation inequalities in learning theory which measures the effect of random sampling in the probabilistic sense. The following two propositions can be proved using the arguments given in Step 2.1. of \cite[Thm.~4]{Caponnetto}.

\begin{proposition}\label{main.bound}
Suppose Assumptions~\ref{fp}--\ref{assmpt1} hold true, then for~$m \in \NN$ and~$0<\eta<1$, each of the following estimate holds with the confidence~$1-\eta$,
\begin{equation*}
\Psi = \Psi(\la) := \norm{(\tp+\la I)^{-1/2}\bx^*(\yy-\sx A(\fp))}{\HH} \leq 2\paren{\frac{\kappa M}{m\sqrt{\la}}+\sqrt{\frac{\Sigma^2\NT(\la)}{m}}}\log\left(\frac{2}{\eta}\right),
\end{equation*}

\begin{equation*}
\Upsilon= \Upsilon(\la): =\norm{(\tp+\la I)^{-1/2}(\tp-\tx)}{HS} \leq 2\paren{\frac{\kappa^2}{m\sqrt{\la}}+\sqrt{\frac{\kappa^2\NT(\la)}{m}}}\log\left(\frac{2}{\eta}\right),
\end{equation*}

\begin{equation*}
\norm{\tp-\tx}{HS} \leq 2\paren{\frac{\kappa^2}{m}+\frac{\kappa^2}{\sqrt{m}}}\log\left(\frac{2}{\eta}\right)
\end{equation*}
and
\begin{equation*}
\Lambda = \Lambda(\la) := \norm{(\lp+\la I)^{-1/2}(\lx-\lp)}{HS}\leq 2\left(\frac{\tilde{\kappa}^2}{m\sqrt{\la}}+\sqrt{\frac{\tilde{\kappa}^2\NL(\la)}{m}}\right)\log\left(\frac{2}{\eta}\right).
\end{equation*}
\end{proposition}

In the following proposition, the probabilistic estimate of the first term can be established under the condition~\eqref{l.la.condition} on the regularization parameter~$\la$ and sample size~$m$. Then we obtain the last two estimates using \cite[Prop.~A.2]{Blanchard2019}.

\begin{proposition}\label{I1}
Suppose Assumption~\ref{assmpt1} and the condition~\eqref{l.la.condition} hold true. Let~$\zeta : \RR^+ \to \RR^+$ be a nondecreasing and sub-linear function, then for~$m \in \NN$ and~$0<\eta<1$, each of the following estimates hold with the confidence~$1-\eta$,

\begin{equation*}
\Upsilon= \norm{(\tp+\la I)^{-\frac{1}{2}}(\tp-\tx)}{HS}\leq \sqrt{\la}2\kappa(2\kappa+1)\log\paren{\frac{2}{\eta}},
\end{equation*}

\begin{align*}
\Xi^s = \Xi^s(\la) :=\norm{(\tx+\la I)^{-s}(\tp+\la I)^s}{\lh}\leq  \paren{\frac{\Upsilon}{\sqrt{\la}}+1}^{2s} \leq & \paren{(2\kappa+1)^2\log\paren{\frac{2}{\eta}}}^{2s} 
\end{align*}
for~$0\leq s \leq 1$ and
\begin{align*}
\Xi^\zeta = \Xi^\zeta(\la) := \norm{\paren{\frac{1}{\zeta}}(\tx+\la I)\zeta(\tp+\la I)}{\lh} \leq & \paren{\frac{\Upsilon}{\sqrt{\la}}+1}^2 
\leq  \paren{(2\kappa+1)^2\log\paren{\frac{2}{\eta}}}^{2}.
\end{align*}
\end{proposition}

\begin{lemma}\label{eq:lemma}
Suppose Assumption~\ref{ass:link} holds true. Let~$\ga$ be any regularization with residual function~$\ra$. Then for~$\upsilon(t)=t/\varrho(t)$, we have that
  \begin{equation}
\norm{L^{-1}\ra(\tx) L}{\lh} \leq 1+(B + D) \paren{\Xi^\varrho  \Xi^\upsilon   + \Xi \varrho(\la)(\varrho(\la)+1)\frac{\Lambda}{\sqrt{\la}}}.
  \end{equation}
\end{lemma}
\begin{proof}
For~$\lx=A^*\sx^*\sx A$ and~$\lp=A^*\ip^*\ip A$ with the fact that~$\tx - \tp = L^{-1} \paren{\lx - \lp} L^{-1}$, the proof will be based on the following decomposition
\begin{align*}
L^{-1}\ra(\tx)L = & I - L^{-1}\ga(\tx)\tp L + L^{-1}\ga(\tx)L^{-1}(\lp-\lx) \\
 = & I-L^{-1}\ga(\tx) \tp L  + \la L^{-1}\ga(\tx)L^{-1}(\lp+\la I)^{-1}\paren{\lp - \lx} \\
&+  L^{-1}\ga(\tx)L^{-1}\lp (\lp+\la I)^{-1}\paren{\lp - \lx},
\end{align*}
and this yields the estimate
  \begin{align}
    \norm{L^{-1} \ra(\tx)L}{\lh}  \leq & 1 + \norm{L^{-1}\ga(\tx) \tp L}{\lh} + \la\norm{L^{-1}\ga(\tx)L^{-1}(\lp+\la I)^{-1}\paren{\lp - \lx}}{\lh} \notag\\
   & + \norm{L^{-1}\ga(\tx)L^{-1}\lp (\lp+\la I)^{-1}\paren{\lp - \lx}}{\lh} \notag\\
    & = 1+ I_{1} + \la I_{2} + I_3.
   \end{align}

We observe that~$\norm{\ga(\tx)(\tx +\la I)}{\lh} \leq B+D$. For the function~$\upsilon(t)=t/\varrho(t)$, we can bound~$I_{1}$ as
\begin{align*}
  \norm{L^{-1}\ga(\tx) \tp L}{\lh} \leq & \norm{L^{-1}\frac{1}{\varrho}\paren{\tp + \la I}}{\lh}\norm{\varrho\paren{\tp + \la I}\frac{1}{\varrho}\paren{\tx +\la I}}{\lh} \norm{\ga(\tx) \paren{\tx +\la I}}{\lh}\\
&\times \norm{ \paren{\tx +\la I}^{-1} \varrho\paren{\tx +\la I} \paren{\tp + \la I} \frac{1}{\varrho}\paren{\tp + \la I}}{\lh}\\                             
&\times\norm{\varrho\paren{\tp + \la I}\paren{\tp + \la I}^{-1}\tp L}{\lh}\\
  & \leq \Xi^\upsilon\Xi^\varrho (B+D) \norm{L^{-1}\frac{1}{\varrho}\paren{\tp+\la I}}{\lh}  \norm{\varrho\paren{\tp + \la I}\paren{\tp + \la I}^{-1}\tp L}{\lh}.
\end{align*}

It remains to bound the second and third factors. From Proposition~\ref{prop:Heinz} we find that
$$
\norm{L^{-1} \frac 1 \varrho(\tp +\la I)}{\lh} \leq \norm{\varrho(\tp) \frac 1 \varrho(\tp +\la I)}{\lh}\leq \norm{\varrho(\tp + \la I) \frac 1 \varrho(\tp+\la I)}{\lh}= 1.
$$

Again, under Assumption~\ref{ass:link} we find that
$$
\norm{\varrho\paren{\tp + \la I}\paren{\tp + \la I}^{-1}\tp L}{\lh} \leq \norm{\varrho\paren{\tp + \la I}\paren{\tp + \la I}^{-1}\tp \frac 1 \varrho(\tp)}{\lh} \leq 1,
$$
which finally yields that~$I_{1} \leq \Xi^\upsilon\Xi^\varrho(B+D)$.

The terms~$I_2$,~$I_3$ can be bounded as
   \begin{align}
I_2 = & \norm{L^{-1}\ga(\tx)L^{-1}(\lp+\la I)^{-1}\paren{\lp - \lx}}{\lh} \\ \nonumber
\leq & \frac{1}{\sqrt{\la}} \norm{L^{-1}\ga(\tx)L^{-1}}{\lh}\norm{(\lp+\la I)^{-1/2}\paren{\lp - \lx}}{\lh} \\ \nonumber
\leq & \frac{1}{\sqrt{\la}} \norm{\varrho(\tp)\ga(\tx)\varrho(\tp)}{\lh}\norm{(\lp+\la I)^{-1/2}\paren{\lp - \lx}}{\lh}  \\ \nonumber
\leq & \frac{1}{\sqrt{\la}} \norm{\varrho(\tp)(\tp+\la I)^{-1/2}}{\lh}^2\norm{(\tp+\la I)(\tx+\la I)^{-1}}{\lh}\norm{\ga(\tx)(\tx+\la I)}{\lh}\\ \nonumber
&\times\norm{(\lp+\la I)^{-1/2}\paren{\lp - \lx}}{\lh} \\ \nonumber
\leq & \frac{\varrho^2(\la)}{{\la}^{3/2}} (B+D)\Xi\norm{(\lp+\la I)^{-1/2}\paren{\lp - \lx}}{\lh}
  \end{align}
and 
   \begin{align}
I_3 = & \norm{L^{-1}\ga(\tx)L^{-1}\lp (\lp+\la I)^{-1}\paren{\lp - \lx}}{\lh} \\ \nonumber
\leq & \norm{L^{-1}\ga(\tx)\bp^*}{\LL\to\HH}\norm{\ip A (\lp+\la I)^{-1}\paren{\lp - \lx}}{\HH\to\LL} \\ \nonumber
= & \norm{L^{-1}\ga(\tx)\tp^{1/2}}{\lh}\norm{\lp^{1/2} (\lp+\la I)^{-1}\paren{\lp - \lx}}{\lh} \\ \nonumber
\leq & \norm{\varrho(\tp)\ga(\tx)\tp^{1/2}}{\lh}\norm{(\lp+\la I)^{-1/2}\paren{\lp - \lx}}{\lh} \\ \nonumber
\leq & \frac{\varrho(\la)}{{\la}^{1/2}} (B+D)\Xi\norm{(\lp+\la I)^{-1/2}\paren{\lp - \lx}}{\lh}.
  \end{align}

This complete the proof.
\end{proof}

\bibliography{library}
\bibliographystyle{plain}
\end{document}